\documentclass{article}

\usepackage[latin1]{inputenc}  
\usepackage[T1]{fontenc}
\usepackage{amsmath,amssymb,dcolumn,amsthm}
\usepackage{graphicx,color,transparent}

\usepackage{algorithm}
\usepackage{algorithmic}

\graphicspath{{./figs/}}
\usepackage{float}
\makeatletter
\renewcommand{\maketag@@@}[1]{\hbox{\m@th\normalsize\normalfont#1}}%
\makeatother

\usepackage{tikz}
\usetikzlibrary{arrows}
\tikzstyle{block} = [draw, draw=black, line width = 1pt, rectangle,
    rounded corners=4pt,
    minimum height=3em, text width=.7\columnwidth,
    inner sep = 10pt]

\usepackage{ian_misc,ian_math,ian_opt,ian_abbrvs,ian_color,ian_par_mpc,ian_par_ric,ian_par_mhe,axeIEEEarticle}

\usepackage{ragged2e}
\usepackage{cite}

\title{\LARGE \bf An $\Ordo{\log{N}}$ Parallel Algorithm for Newton Step Computations\\ with Applications to Moving Horizon Estimation}

\author{Isak Nielsen and Daniel Axehill%
  \thanks{I. Nielsen and D. Axehill are with the Division of Automatic
    Control, Linköping University, SE-58183 Linköping, Sweden, \texttt{isak.nielsen@liu.se, daniel.axehill@liu.se}.}}
    
\begin{document}
\maketitle
\thispagestyle{empty}
\pagestyle{empty}

\begin{abstract}
In Moving Horizon Estimation (\MHE) the computed estimate is found by solving a constrained finite-time optimal estimation problem in real-time at each sample in a receding horizon fashion. The constrained estimation problem can be solved by, \eg, interior-point (\IP) or active-set (\AS) methods, where the main computational effort in both methods is known to be the computation of the search direction, \ie, the Newton step. This is often done using generic sparsity exploiting algorithms or serial Riccati recursions, but as parallel hardware is becoming more commonly available the need for parallel algorithms for computing the Newton step is increasing. In this paper a tailored, non-iterative parallel algorithm for computing the Newton step using the Riccati recursion is presented. The algorithm exploits the special structure of the Karush-Kuhn-Tucker system for the optimal estimation problem. As a result it is possible to obtain logarithmic complexity growth in the estimation horizon length, which can be used to reduce the computation time for \IP and \AS methods when applied to what is today considered as challenging estimation problems. Promising numerical results have been obtained using an \ansic implementation of the proposed algorithm running on true parallel hardware. Beyond \MHE, due to similarities in the problem structure, the algorithm can be applied to various forms of on-line and off-line smoothing problems.
\end{abstract}

%

\newpage
\section{Introduction}
\label{sec:intro}
One of the most widely used advanced control strategies in industry
today is Model Predictive Control (\MPC). In each sample, the \MPC
strategy requires the solution of a constrained finite-time optimal
control (\cftoc) problem on-line, which creates a need for efficient
optimization algorithms. It is well-known that the resulting
optimization problem obtains a special structure that can be exploited
to obtain high-performance linear algebra for computing Newton steps
in various setups, see \eg
\cite{jonson83:thesis,rao98:_applic_inter_point_method_model_predic_contr,hansson00:_primal_dual_inter_point_method,bartlett02:_quadr,vandenberghe02:_robus_full,akerblad04:_effic,axehill06:_mixed_integ_dual_quadr_progr_tail_mpc,axevanhan07:_relax_applic_mipc_compa_and_eff_compu,axehill08:thesis,axehill08:_dual_gradien_projec_quadr_progr,diehl09:_nonlin_model_predic_contr,domahidi12:_effic_method_multis_probl_arisin}. A
problem which turns out to have similar problem structure is the
Moving Horizon Estimation (\MHE) problem,
\cite{rao00:thesis,jorgensen04:thesis}. In \MHE, the state-estimate is
again obtained as the solution to a highly structured optimization
problem solved on-line in a receding horizon fashion. In the same
spirit as \MPC adds the possibility for optimal control under
constraints, \MHE adds the possibility for optimal estimation under
constraints. It has been shown that problem structure can be exploited
also for this application
in~\cite{rao00:thesis,jorgensen04:thesis,haverbeke09:structure_expl_ip_for_mhe}. The
optimization problem that is solved on-line in \MHE can be shown to
have a similar structure to the one in so-called smoothing,
\cite{rao00:thesis,kailath2000linear}. In smoothing,
measurements are available along the entire time window of estimation,
which means that non-causal estimation can be performed.
From a computational point of view, \MHE can be interpreted as
repeatedly solving smoothing problems in a receding horizon fashion
and only the last state estimate is actually returned as an estimate
(analogously to that only the first computed control signal is applied in
\MPC). Depending on the type of system and problem formulation, the
\MHE problem can be of different types. \MHE can be applied to linear,
nonlinear or hybrid systems. Often, the computational effort spent
when solving the resulting optimization problem boils down to solving
Newton-system-like equations that can be associated with an unconstrained
finite-time optimal control (\uftoc) problem,
\cite{jorgensen04:thesis}.

In recent years, the need for parallel algorithms for solving control and estimation
problems has increased. While much effort in research has been spent on
this topic for \MPC,~\cite{constantinides2009tutorial}, parallelism for estimation is a less explored field. For the \MPC application, an extended Parallel Cyclic
Reduction algorithm is introduced in~\cite{soudbakhsh2013parallelized} which is used to reduce the computations to smaller systems of equations that are solved in
parallel. The computational complexity of this algorithm is reported
to be $\Ordo{\log N}$, where $N$ is the prediction horizon. In
\cite{ZhuParallelNP} and \cite{reuterswardLic} a time-splitting
approach to split the prediction horizon into blocks is adopted. The
subproblems are solved in
parallel using Schur complements, and the common variables are computed
by solving a dense system of equations serially. In~\cite{o2012splitting} a
splitting method based on Alternating Direction Method of Multipliers
(\admm) is used, where some steps of the algorithm can be computed in
parallel. In \cite{stathopoulos2013hierarchical} an iterative
three-set splitting quadratic programming (\QP) solver is
developed. In this method several simpler subproblems are computed in
parallel and a consensus step using \admm is performed to obtain the
final
solution. A parallel coordinate descent method for solving \MHE problems is proposed in~\cite{poloni13:par_mhe}. In~\cite{nielsen14:parallel_mpc_ifac,nielsen14_parallel_mpc_arxiv}
a tailored algorithm for solving the Newton step directly
(non-iteratively) in parallel for \MPC is presented. In that work
several subproblems are solved parametrically in parallel by
introducing constraints on the terminal states, but the structure
is not exploited when the subproblems are
solved. In~\cite{nielsen15:parallel_factorization_unp}, it is shown
how the Newton step can be computed in parallel while simultaneously
exploiting problem
structure. In~\cite{sina15:distributed_message_passing} a generic
message-passing parallel algorithm for distributed optimization with
applications to, \eg, control and estimation, is presented.

The main contribution in this paper is to extend the use of the
parallel structure exploiting numerical algorithms for computing
Newton steps for \MPC presented in
\cite{nielsen15:parallel_factorization_unp,nielsen15:licthesis} to the
\MHE and smoothing problems. Furthermore, the performance of the
algorithm when solving \MHE problems is illustrated using an \ansic implementation of the
proposed algorithm that is executed truly in parallel on a physical
computer cluster. The proposed algorithm can replace
existing serial algorithms at the computationally most demanding step
when solving various forms of finite horizon optimal estimation
problems in practice.  The algorithm is tailored for \MHE problems and
exploit the special structure of the \KKT system for such
problems. The classical serial Riccati method exploits the causality
of the problem and for that reason it is not obvious that it can be
split and parallelized in time, especially without involving an
iterative consensus step. The main idea is to exploit the problem
structure in time and divide the \MHE problem in smaller subproblems
along the prediction horizon. Consensus is reached directly
(non-iteratively) by solving a master problem. This overall structure
is similar to what is done
in~\cite{nielsen15:parallel_factorization_unp}, but the result is here
extended to the \MHE problem. A more detailed presentation of the work
in~\cite{nielsen14:parallel_mpc_ifac}
and~\cite{nielsen15:parallel_factorization_unp} is given
in~\cite{nielsen15:licthesis}. 

In this paper $\posdefmats^n$ ($\possemidefmats^n$) denotes symmetric
positive (semi) definite matrices with $n$ columns, $\intset{i}{j}
\triangleq \braces{i,i+1,\hdots,j}$ and symbols in sans-serif font
(\eg $\timestack{x}$) denote vectors or matrices of stacked
components. Furthermore, $I$ denotes the identity matrix of appropriate
dimension, and the product operator is defined as
\begin{equation}
\prod_{t=t_1}^{t_2} A_t \triangleq \begin{cases} 
A_{t_2} \cdots A_{t_1}, \; t_1 \leq t_2 \\
I, \; t_1 > t_2.
\end{cases}
\end{equation}



\section{Problem Formulation}
\label{sec:prob_form}

The \MHE problem is solved by solving the corresponding inequality constrained optimization problem. This can be done using different types of methods, where some common ones are primal and primal-dual interior-point (\IP) methods and active-set (\AS) methods. In these types of methods the main computational effort is spent when computing the search directions,~\cite{boyd04:_convex_optim,nocedal06:num_opt}, which is interpreted as solving a sequence of equality constrained \QP problems,~\cite{rao98:_applic_inter_point_method_model_predic_contr,jorgensen04:thesis,haverbeke09:structure_expl_ip_for_mhe}. The equality constrained convex \QP problem has the structure
\begin{equation}
\minimizes{&\frac{1}{2} \parens{x_0 - \xtil{0}}^T \inv{\Ptil{0}} \parens{x_0-\xtil{0}}+ \\ & \frac{1}{2} \sum_{k=0}^{N^{\textrm{mhe}}} \begin{bmatrix}
\w_k - \wtil_k \\ \vv_k - \vtil_k
\end{bmatrix}^T \inv{ \begin{bmatrix}
\Qwtil{k} & \! \Qwvtil{k} \\ \Qwvtil{k}^T & \! \Qvtil{k}
\end{bmatrix}  }\begin{bmatrix}
\w_k - \wtil_k \\ \vv_k - \vtil_k
\end{bmatrix} }{ \timestack{x},\timestack{w},\timestack{v}}{&x_{k+1} = A_k x_k + B_k \w_k + a_k, \; k \in \intset{0}{N^{\textrm{mhe}}} \\ &y_k = C_k x_k + \vv_k + d_k, \; k \in \intset{0}{N^{\textrm{mhe}}},} \label{eq:org_mhe}
\end{equation}
where $x_k \in \realnums{\nx}$ is the state, $\w_k \in \realnums{\nww} $ is the process noise, $\vv_k \in \realnums{\ny}$ is the sensor noise and $y_k \in \realnums{\ny}$ is the measured output,~\cite{jorgensen04:thesis}. $\xtil{0}$ and $\Ptil{0}$ are the initial state estimate and covariance matrix, respectively, and $\wtil_{k}$ and $\vtil_{k}$ are the nominal values for $\w_k$ and $\vv_k$, respectively. Here, the problem~\eqref{eq:org_mhe} is considered to be a deterministic optimization problem. However, a stochastic interpretation of this problem is found in, \eg,~\cite{kailath2000linear}. It is shown in, \eg{},~\cite{jorgensen04:thesis} that the \QP problem~\eqref{eq:org_mhe} can equivalently be written in the form of a \uftoc problem. This is done by eliminating the variable $\vv_k$ from the objective function using the measurement equation, and by defining a new state variable $x_{-1} \triangleq \xtil{0}$ and its corresponding process noise $\w_{-1} \triangleq x_0 - \xtil{0}$, which gives the relation $x_0 = x_{-1} + \w_{-1}$. Furthermore, by shifting time-indices by introducing $t \triangleq k+1$ and $N \triangleq N^{\textrm{mhe}}+2$, the problem~\eqref{eq:org_mhe} can equivalently be written in the form of a \uftoc problem, which here will be denoted $\MPCprob{N}$, \ie{},
\begin{equation}
 \label{eq:org_eqc_problem} \minimizes{
    &\sum^{N-1}_{t=0}\left(\frac{1}{2}\begin{bmatrix}
    x_t \\ \w_t
    \end{bmatrix}^T
    Q_t\begin{bmatrix}
    x_t \\ \w_t
\end{bmatrix} +  l_t^T \begin{bmatrix}
x_t \\ \w_t
\end{bmatrix} + c_t \right) \\
     &+\frac{1}{2}x^T_NQ_{x,N}x_N + l_N^Tx_N + c_N}
  {\timestack{x},\timestack{\w}}
  {&x_0 = \xtil{0} \\
    &x_{t+1} = A_tx_t + B_t\w_t + a_t, \; t \in \intset{0}{N-1}.} 
\end{equation}
For $t=0$ and $t=N$, the problem matrices are given by
\begin{subequations}
\begin{align}
&Q_0 = \begin{bmatrix}
Q_{x,0} & Q_{xw,0} \\ Q_{xw,0}^T & Q_{w,0}
\end{bmatrix} \triangleq \begin{bmatrix}
0 & 0 \\ 0 & \inv{\Ptil{0}}
\end{bmatrix}, \; l_0 \triangleq 0, \; c_0 \triangleq 0, \\& A_0 \triangleq I, \; B_0 \triangleq I, a_0 \triangleq 0, \; Q_{x,N} \triangleq 0, \; \lin{x,N} \triangleq 0, \; c_N \triangleq 0.
\end{align}
\end{subequations}
By defining $\ytil{t} \triangleq y_t - d_t - \vtil_{t}$ and
\begin{equation}
\begin{bmatrix}
W_t & S_t \\ S_t^T & V_t
\end{bmatrix} \triangleq \inv{ \begin{bmatrix}
\Qwtil{k} & \Qwvtil{k} \\ \Qwvtil{k}^T & \Qvtil{k}
\end{bmatrix}  },
\end{equation}
the problem matrices for $t \in \intset{1}{N-1}$ are given by
\begin{subequations}
\begin{align}
&Q_t = \begin{bmatrix}
Q_{x,t} & Q_{xw,t} \\ Q_{xw,t}^T & Q_{w,t}
\end{bmatrix} \triangleq \begin{bmatrix}
C_t^T V_t C_t & - C_t^T S_t \\ -S_t^T C_t & W_t
\end{bmatrix}\in \possemidefmats^{\nx+\nw}, \\& l_t \triangleq \begin{bmatrix}
\lin{x,t}  \\ \lin{\w,t}
\end{bmatrix}  = \begin{bmatrix}
C_t^T \parens{S_t \wtil_t - V_t \ytil{t} } \\ S_t^T \ytil{t} - W_t \wtil_{t}
\end{bmatrix} , \\& c_t \triangleq \frac{1}{2}  \wtil_{t}^T W_t \wtil_{t} -  \ytil{t}^T S_t \wtil_{t} + \frac{1}{2}\ytil{t}^T V_t \ytil{t}.
\end{align}
\end{subequations}

For the derivation of the problem matrices, see, \eg,~\cite{jorgensen04:thesis}.


\begin{remark}
In both \IP and \AS methods the solution to the original constrained \MHE problem is obtained by solving a sequence of \uftoc problems in the form in~\eqref{eq:org_eqc_problem}. The number of problems in this sequence is independent of how these \uftoc problems are solved. 
Since the main computation time is consumed when the \uftoc problems
are solved, the overall relative performance gain for solving the
entire sequence of problems in order to solve the constrained \MHE problem is roughly the same as the relative performance gain obtained when solving a single \uftoc problem.
\end{remark}


\section{Serial Riccati Recursion}
\label{sec:ric_rec}
The optimal solution to the \uftoc problem~\eqref{eq:org_eqc_problem} is computed by solving the set of linear equations given by the associated Karush-Kuhn-Tucker (\KKT) system. For this problem structure, the \KKT system has a very special form that is almost block diagonal and it is well known that it can be factored efficiently using a Riccati factorization~\cite{axehill08:thesis}. The Riccati factorization is used to factor the \KKT coefficient matrix, followed by backward and forward recursions to compute the primal and dual variables. The computational complexity when solving the \KKT system using the Riccati recursion is reduced from roughly $\Ordo{N^2}-\Ordo{N^3}$ to $\Ordo{N}$ compared to solvers that do not exploit sparsity. 
The Riccati recursion is given by algorithms~\ref{alg:factorization}-\ref{alg:fwd_rec}, where $F_t,P_t\in\possemidefmats^{n_x}$, $G_t\in\posdefmats^{\nw}$, ${H_t\in\realnums{n_x\times \nw}}$ and ${K_t\in\realnums{\nw \times
  n_x}}$,~\cite{axehill08:thesis}.
For more background information on Riccati factorizations, see, e.g.,~\cite{jonson83:thesis},~\cite{rao98:_applic_inter_point_method_model_predic_contr} or~\cite{axehill08:thesis}.
%
%
\begin{algorithm}[h!]
  \caption{Riccati factorization} \label{alg:factorization}
  \begin{algorithmic}[1]
    \STATE $P_N := Q_{x,N}$
    \FOR{$t=N-1,\ldots,0$}
    \STATE $F_{t+1} := Q_{x,t} + A^T_tP_{t+1}A_t$\label{alg:factorization:line:F} \\
    \STATE $G_{t+1} := Q_{\w,t} + B^T_tP_{t+1}B_t$ \label{alg:factorization:line:G}\\
    \STATE $H_{t+1} := Q_{x\w,t} + A^T_tP_{t+1}B_t$\label{alg:factorization:line:H} \\
    \STATE Compute and store a factorization of $G_{t+1}$.\label{alg:factorization:line:factorize_G}
    \STATE Compute a solution $K_{t+1}$ to \\ 
    $G_{t+1}K_{t+1} = -H^T_{t+1}$ \label{alg:factorization:line:GK}\\
    \STATE $P_{t} := F_{t+1} - K^T_{t+1}G_{t+1}K_{t+1}$\label{alg:factorization:line:P}
    \ENDFOR
  \end{algorithmic}
\end{algorithm}
\begin{algorithm}[tbp!]
  \caption{Backward recursion}
  \label{alg:bwd_rec}
  \begin{algorithmic}[1]
    \STATE $\Psi_N := -\lin{x,N}, \; \cb{N} := c_N$
    \FOR{$t = N-1,\hdots,0$}
    \STATE Compute a solution $k_{t+1}$ to \\
    	$G_{t+1}k_{t+1} = \parens{B^T_t\Psi_{t+1} -
        \lin{\w,t} - B_t^T P_{t+1}a_t}$ \label{alg:bwd_rec:line:Gk}
    \STATE $\Psi_t := A^T_t\Psi_{t+1}
    - H_{t+1}k_{t+1} - \lin{x,t} - A_t^T P_{t+1}a_t$ \\
    \STATE $\cb{t} := \cb{t+1} + \frac{1}{2}a_t^TP_{t+1}a_t - \Psi_{t+1}^Ta_t$ \\ \hspace{4ex} $- \frac{1}{2}k_{t+1}^TG_{t+1}k_{t+1} + c_t$
    \ENDFOR
  \end{algorithmic}
\end{algorithm}
\begin{algorithm}[htbp!]
  \caption{Forward recursion}
  \label{alg:fwd_rec}
  \begin{algorithmic}[1]
  	\STATE $x_0 := \xtil{0}$
    \FOR{$t = 0,\hdots,N-1$}
	\STATE $\w_t := k_{t+1} + K_{t+1}x_t$
    \STATE $x_{t+1} := A_tx_t +
    B_t \w_t + a_t$
    \STATE $\lambda_t := P_tx_t -\Psi_t$ 
    \ENDFOR
    \STATE $\lambda_N := P_Nx_N-\Psi_N$
  \end{algorithmic}
\end{algorithm}



\section{Problem Decomposition and Reduction}
\label{sec:prob_red}

By examining algorithms~\ref{alg:factorization} and~\ref{alg:bwd_rec}, it can be seen that given $P_{t_{i+1}}$, $\Psi_{t_{i+1}}$ and $\cb{t_{i+1}}$ the factorization and backward recursion can be computed for $0 \leq t \leq  t_{i+1} $. Furthermore, if these algorithms have been executed, it follows from Algorithm~\ref{alg:fwd_rec} that given $x_{t_i}$ the forward recursion can be computed for the interval ${t_i \leq t \leq t_{i+1}}$. Hence, provided that $P_{t_{i+1}}$, $\Psi_{t_{i+1}}$, $\cb{t_{i+1}}$ and $x_{t_i}$ are known for $i \in \intset{0}{p}$ for some $p$, it is possible to compute the Riccati recursion and the primal and dual solution in each interval ${t_{i} \leq t \leq t_{i+1}}$ with $i \in \intset{0}{p}$ independently from the other intervals. This property will be used to decompose the problem.

The decomposition of the time-horizon is similar to what is done in partial condensing, which is introduced in~\cite{axehill15:_contr_mpc}. In partial condensing, state variables in several batches along the prediction horizon are eliminated, and the resulting control problem can be interpreted as a problem with shorter prediction horizon but larger control input/process noise dimension. In the parallel approach in this paper, due to the property described above, the partial condensing of the independent batches can be performed in parallel. Furthermore, by utilizing the problem structure in the batches it is possible to also reduce the control input/process noise dimension.


%


\subsection{Divide into independent intervals}
\label{subsec:sub_prob}


Decompose the \uftoc problem~\eqref{eq:org_eqc_problem} by dividing the prediction horizon into $\p+1$ intervals, or {batches}. This is done by introducing the batch-wise variables $\timestack{x_i}$ and $\timestack{\w_i}$ as
\begin{align}
\timestack{x_i} &= \begin{bmatrix}
x_{0,i}^T & \cdots & x_{N_i,i}^T
\end{bmatrix}^T \triangleq \begin{bmatrix}
x_{t_{i}}^T & \cdots & x_{t_{i+1}}^T
\end{bmatrix}^T, \\
\timestack{\w_i} &= \begin{bmatrix}
\w_{0,i}^T & \cdots & \w_{N_i-1,i}^T
\end{bmatrix}^T \triangleq \begin{bmatrix}
\w_{t_{i}}^T & \cdots & \w_{t_{i+1}-1}^T
\end{bmatrix}^T, 
\end{align}
where $N_i$ is the length of batch $i$, $t_0 = 0$ and $x_{N_i,i} = x_{0,i+1}$. 

By following the reasoning in the introduction of this section it is possible to compute the Riccati recursion and the optimal value in batch $i$ if ${\xh{i} \triangleq x_{\tint{i}}}$, $\Prh{i+1} \triangleq P_{\tint{i+1}}$, $\Psirh{i+1} \triangleq \Psi_{\tint{i+1}}$ and $\crh{i+1} \triangleq \cb{\tint{i+1}}$ are known. Hence, if these variables are known for all batches $i \in \intset{0}{p}$, the solution to the original problem~\eqref{eq:org_eqc_problem} can be computed from $p+1$ independent subproblems in the \uftoc form

\begin{equation}
\minimizes{&\sum_{t=0}^{N_i-1}\left(\frac{1}{2}\begin{bmatrix}
    x_{t,i} \\ \w_{t,i}
    \end{bmatrix}^T Q_{t,i} \begin{bmatrix}
    x_{t,i} \\  \w_{t,i}
\end{bmatrix} + \lin{t,i}^T \begin{bmatrix}
    x_{t,i} \\ \w_{t,i}
\end{bmatrix} + c_{t,i}\right) \\&+\frac{1}{2}x_{N_i,i}^T\Prh{i+1}x_{N_i,i} - \Psirh{i+1}^T x_{N_i,i} + \crh{i+1}}{\timestack{x_i},\timestack{u_i}}{x_{0,i} &= \xh{i} \\ x_{t+1,i}&=A_{t,i}x_{t,i}+B_{t,i}\w_{t,i}+a_{t,i}, \; t\in \intset{0}{N_i-1},} \label{eq:sub_problem_dscr}
\end{equation}
using $p+1$ individual Riccati recursions. Here $ Q_{t,i}$, $\lin{t,i}$, $c_{t,i}$, $A_{t,i}$, $B_{t,i}$ and $a_{t,i}$ are defined consistently with $\timestack{x_i}$ and $\timestack{\w_i}$. 



\subsection{Eliminate local variables in a subproblem}
\label{subsec:red_sub_prob}
A detailed description of the elimination of local variables in the subproblems is given in~\cite{nielsen15:licthesis,nielsen15:parallel_factorization_unp}. However, a shortened version including all the main steps are given here for completeness.
It is shown that even when $\Prh{i+1}$, $\Psirh{i+1}$ and $\crh{i+1}$ are not known in~\eqref{eq:sub_problem_dscr}, it is possible to eliminate local variables and reduce the sizes of the individual subproblems. The core idea with this approach is that the unknowns $\Prh{i+1}$ and $\Psirh{i+1}$ will indeed influence the solution of the subproblem, but as is shown in~\cite{nielsen15:licthesis,nielsen15:parallel_factorization_unp}, the resulting degree of freedom is often very limited compared to the dimension of the full vector $\timestack{\w_i}$. The constant $\crh{i+1}$ affects the optimal value of the cost function but not the solution. The elimination of variables can be done separately for the $\p+1$ subproblems, which opens up for a structure that can be solved in parallel. 
%
%
In the remaining part of this section the subindices $i$ in~\eqref{eq:sub_problem_dscr} are omitted for notational brevity, \ie, $\Psirh{i+1}$ is written $\Psirh{}$ etc. 

It will now be shown how the structure in the subproblem~\eqref{eq:sub_problem_dscr} can be exploited to eliminate local variables in an interval. The elimination of local variables and reduction of the subproblems will be simplified by using a preliminary feedback policy which is computed using the Riccati recursion. The use of this preliminary feedback is in principle not necessary, but it will later be seen that the main computationally demanding key computations can be performed more efficiently by using it. To compute the preliminary feedback, let the \uftoc problem~\eqref{eq:sub_problem_dscr} with unknown $\Prh{}$, $\Psirh{}$ and $\crh{}$ be factored and solved for the preliminary choice ${\Prh{}=0}$, $\Psirh{} = 0$ and $\crh{} = 0$ using algorithms~\ref{alg:factorization} and~\ref{alg:bwd_rec}. The resulting optimal policy for $\Prh{} = 0$ and $\Psirh{}=0$ is then $\w_{0,t} = k_{0,t+1}+K_{0,t+1}x_t$ for $t \in \intset{0}{N-1}$. The subindex ''$0$'' denotes variables that correspond to this preliminary solution. 


It will now be investigated how $\w_t$ and the cost function are affected when $\Prh{} \neq 0$, $\Psirh{} \neq 0$ and $\crh{} \neq 0$. Let the contribution to $\w_t$ from the unknown $\Prh{}$ and $\Psirh{}$ be denoted $\wb{t} \in \mathbb{R}^{\nw}$. Using the preliminary feedback, which is optimal for $\Ph{} = 0$ and ${\Psih{} = 0}$, $\w_t$ can be written
\begin{equation}
\w_{t} = k_{0,t+1}+K_{0,t+1}x_{t}+\wb{t}, \; t \in \intset{0}{N-1}. \label{eq:control_signal_def}
\end{equation}
Note that $\wb{t}$ is an arbitrary $\nw$-vector, hence there is no loss of generality in this assumption. From now on, the policy~\eqref{eq:control_signal_def} is used in the subproblem, and it will be shown that the degree of freedom in $\timestack{\wb{}}$ can be reduced. It is shown in~\cite{nielsen15:licthesis,nielsen15:parallel_factorization_unp} that the \uftoc problem~\eqref{eq:sub_problem_dscr} can be written
\begin{equation}
\minimizes{&\frac{1}{2}x_0^T\Qh{x}x_0 + \lh{x}^Tx_0+\frac{1}{2}\timestack{\wb{}}^T\timestack{\Qb{\wb{}}}\timestack{\wb{}} + \cb{0,0}\\&+\frac{1}{2}x_N^T\Prh{}x_N-\Psirh{}^Tx_N + \crh{}}{x_0,\timestack{\wb{}},x_N}{x_0 &= \xh{} \\ x_N &= \Ah{}x_0+\Sc \timestack{\wb{}} + \ah{},} \label{eq:proof:thm:reduce_subprob:dense_prob}
\end{equation}
when using the preliminary feedback~\eqref{eq:control_signal_def}. Here
\begin{subequations}
\begin{equation}
\Qh{x} \triangleq P_{0,0}, \quad \lh{x} \triangleq -\Psi_{0,0}, \label{eq:Qxh_lxh_def}
\end{equation}
\begin{equation}
\timestack{\Qb{\wb{}}}\triangleq \begin{bmatrix}
G_{0,1} \\
 & \ddots \\
 & & G_{0,N}
\end{bmatrix}, 
\Ah{} \triangleq  \prod_{t=0}^{N-1} \left(A_t + B_t K_{0,t+1} \right), \label{eq:proof:thm:reduce_subprob:QuA}
\end{equation}
\begin{equation}
\timestack{S} \triangleq \begin{bmatrix}
\prod_{t=1}^{N-1} \left(A_t + B_t K_{0,t+1} \right)B_0 & & \hdots & B_{N-1}
\end{bmatrix}, \label{eq:proof:thm:reduce_subprob:B}
\end{equation}
\begin{align}
\ah{} &\triangleq 
 \sum_{\tau=0}^{N-1}\prod_{t=\tau+1}^{N-1} \left(A_t + B_t K_{0,t+1} \right)(a_{\tau}+B_{\tau}k_{0,\tau+1}) ,\label{eq:proof:thm:reduce_subprob:a}
\end{align}
\end{subequations}
and $P_{0,0}$, $\Psi_{0,0}$ and $\cb{0,0}$ are computed by algorithms~\ref{alg:factorization} and~\ref{alg:bwd_rec} with the preliminary choice $\Prh{} = 0$, $\Psirh{} = 0$ and $\crh{} = 0$. 

 The problem~\eqref{eq:proof:thm:reduce_subprob:dense_prob} is a \uftoc problem with prediction horizon~$1$ and~$N \nw$ control inputs. The equations that define the factorization of the \KKT system of~\eqref{eq:proof:thm:reduce_subprob:dense_prob} are 
\begin{subequations}
\label{eq:parric:reduce_subprob}
\begin{align}
\Fh{} &= P_{0,0} + \Ah{}^T \Prh{} \Ah{}, \label{eq:proof:thm:reduce_subprob:Fbar}\\
\Gb{} &= \Qb{\wb{}}+ \Sc^T \Prh{}\, \Sc \label{eq:proof:thm:reduce_subprob:G}, \quad
\Hbpric{} = \Ah{}^T \Prh{} \, \Sc ,\\ 
\Gb{}\Kb{} &= - \Hbpric{}^T \label{eq:proof:thm:reduce_subprob:GK} , \quad 
\Gb{}\kb{} = \Sc^T\left( \Psirh{} - \Prh{}\ah{} \right). 
\end{align}
\end{subequations}
These can be used to compute the optimal solution of~\eqref{eq:proof:thm:reduce_subprob:dense_prob} to obtain the optimal $\timestack{\bar{\w}}$.
%
%
Using~\eqref{eq:proof:thm:reduce_subprob:G}, the first equation in~\eqref{eq:proof:thm:reduce_subprob:GK} can be written as
\begin{equation}
\left(\Qb{\wb{}}+ \Sc^T \Prh{}\, \Sc \right)\Kb{} = -\Sc^T \Prh{} \Ah{}. \label{eq:proof:thm:reduce_subprob:GK_written_out}
\end{equation}
In~\cite{nielsen15:licthesis,nielsen15:parallel_factorization_unp} it is shown that it is possible to reduce the number of equations and reducing the degree of freedom of $\timestack{\bar \w}$ by exploiting the structure in~\eqref{eq:proof:thm:reduce_subprob:GK_written_out}. To do this, let $U_1 \in \mathbb{R}^{N \nw \times \nuo}$ with $\nuo \leq \nx$ be an orthonormal basis for $\range{\Sc^T}$, \ie, the range space of $\Sc^T$. Then, by introducing $\Kh{} \in \realnums{\nuo \times \nx}$ to parametrize the feedback matrix as
\begin{equation}
\Kb{} = \inv{\timestack{{\bar Q_{\wb{}}}}} U_1 \Kh{},
\end{equation}
and inserting this choice of $\Kb{}$ into~\eqref{eq:proof:thm:reduce_subprob:GK_written_out} gives
\begin{equation}
\parens{U_1 + \Sc^T \Ph{} \, \Sc\inv{\timestack{{\bar Q_{\wb{}}}}} U_1} \Kh{} = -\Sc^T \Prh{} \Ah{}. \label{eq:GK_reduce_step_1}
\end{equation}
Furthermore, by multiplying~\eqref{eq:GK_reduce_step_1} with the full rank matrix $U_1^T \inv{\timestack{{\bar Q_{\wb{}}}}}$ from the left gives
\begin{equation}
\parens{U_1^T \inv{\timestack{{\bar Q_{\wb{}}}}}  U_1 + U_1^T \inv{\timestack{{\bar Q_{\wb{}}}}} \Sc^T \Ph{} \, \Sc \inv{\timestack{{\bar Q_{\wb{}}}}} U_1 } \Kh{} = - U_1^T\inv{\timestack{{\bar Q_{\wb{}}}}} \Sc^T \Ph{} \Ah{}, \label{eq:proof:thm:reduce_subprob:reducedGK_S}
\end{equation}
which is equivalent to the system of equations~\eqref{eq:proof:thm:reduce_subprob:GK_written_out}.

Now, by introducing the variables
\begin{align}
\Qh{\w} &\triangleq U_1^T \inv{\timestack{{\bar Q_{\wb{}}}}} U_1 \in \posdefmats^{\nuo} , \quad
\Bh{} \triangleq \Sc \inv{\timestack{{\bar Q_{\wb{}}}}}  U_1 \in \realnums{\nx \times \nuo}, \label{eq:Bh_def} \\
\Gh{} & \triangleq \Qh{\w} + \Bh{}^T \Prh{} \Bh{}, \quad
\Hh{} \triangleq \Ah{}^T \Prh{} \Bh{}, \label{eq:proof:thm:Gh_Hh_def}
\end{align}
eq.~\eqref{eq:proof:thm:reduce_subprob:reducedGK_S} can be written as
\begin{equation}
\Gh{} \Kh{} = - \Hh{}^T. \label{eq:proof:thm:reduce_subprob:GhKh}
\end{equation}
\begin{remark}
The preliminary feedback in~\eqref{eq:control_signal_def} results in a block-diagonal $\timestack{\Qb{\wb{}}}$ with blocks given by $G_{0,t+1}$ for $t \in \intset{0}{N-1}$. Hence, $\Qh{\w}$ and $\Bh{}$ can be computed efficiently using block-wise computations where the factorizations of $G_{0,t+1}$ from computing $K_{0,t+1}$ can be re-used.
\end{remark}

By using analogous calculations, the structure can be exploited also in the second equation in~\eqref{eq:proof:thm:reduce_subprob:GK} to reduce it to
\begin{equation}
\Gh{} \kh{} = \Bh{}^T \left(\Psirh{} - \Prh{}\ah{}\right),
\end{equation}
with $\kh{} \in \mathbb{R}^{\nuo}$. Hence,~\eqref{eq:parric:reduce_subprob} can equivalently be written as
\begin{subequations}
\label{eq:parric:reduce}
\begin{align}
\Fh{} &= \Qh{x} + \Ah{}^T \Prh{} \Ah{} \label{eq:proof:thm:reduce:F}, \\ 
\Gh{} &= \Qh{\w} + \Bh{}^T \Prh{} \Bh{}, \quad
\Hh{} = \Ah{}^T \Prh{} \Bh{}, \\
\Gh{} \Kh{} &= -\Hh{}^T \label{eq:proof:thm:reduce:GK}, \quad 
\Gh{} \kh{} = \Bh{}^T\left(\Psirh{} - \Prh{}\ah{} \right), 
\end{align}
\end{subequations}
which can be identified as the factorization of the \KKT system of a \uftoc problem in the form~\eqref{eq:proof:thm:reduce_subprob:dense_prob} but with input signal dimension $n_{\hat \w} = \nuo \leq \nx$. Hence~\eqref{eq:parric:reduce} defines the optimal solution to a smaller \uftoc problem. This important result is summarized in Theorem~\ref{thm:reduce_subprob}, which is repeated from~\cite{nielsen15:licthesis,nielsen15:parallel_factorization_unp} (where the subindices $i$ in~\eqref{eq:sub_problem_dscr} are again used).

%
\begin{theorem}
\label{thm:reduce_subprob}
A \uftoc problem given in the form~\eqref{eq:sub_problem_dscr} with unknown $\Prh{i+1}$, $\Psirh{i+1}$ and $\crh{i+1}$ can be reduced to a \uftoc problem in the form
%
%
\begin{equation}
\minimizes{&\frac{1}{2}x_{0,i}^T \Qh{x,i} x_{0,i} + \frac{1}{2}\wh{i}^T \Qh{\w,i} \wh{i} + \lh{x,i}^Tx_{0,i} + \ch{i} \\
&+ \frac{1}{2}x_{N_i,i}^T \Prh{i+1} x_{N_i,i} - \Psirh{i+1}^Tx_{N_i,i} + \crh{i+1}}{x_{0,i},\, x_{N_i,i}, \,\wh{i}}{x_{0,i} &= \xh{i} \\x_{N_i,i} &= \Ah{i}x_{0,i}+\Bh{i}\wh{i}+\ah{i}, }
\end{equation}
where $\xh{i},x_{0,i}, x_{N_i,i} \in \mathbb{R}^{\nx}$ and $\wh{i} \in \mathbb{R}^{\nw}$, with $\nw \leq \nx$. $\Ah{i}$ and $\ah{i}$ are defined in~\eqref{eq:proof:thm:reduce_subprob:QuA} and~\eqref{eq:proof:thm:reduce_subprob:a}, respectively, and $\Qh{x,i}$, $\Qh{\w,i}$, $\lh{x,i}$ and $\Bh{i}$ are given by~\eqref{eq:Qxh_lxh_def} and~\eqref{eq:Bh_def}, and $\ch{i} \triangleq \cb{0,0}$ where $\cb{0,0}$ is defined as in~\eqref{eq:proof:thm:reduce_subprob:dense_prob}.
\end{theorem}
\begin{proof}
The proof is given in~\cite{nielsen15:parallel_factorization_unp} and~\cite{nielsen15:licthesis}.
\end{proof}

To avoid computing the orthonormal basis $U_1$ in practice a transformation $\Kh{} = T\hat L$, where $T \in \mathbb{R}^{\nuo \times \nx}$ has full rank and $U_1T = \Sc^T$, can be used. By using this choice of $\Kh{}$ in~\eqref{eq:proof:thm:reduce_subprob:reducedGK_S} and then multiplying from the left with $T^T$, the matrices $\Qh{\w}$, $\Bh{}$ and~\eqref{eq:proof:thm:reduce_subprob:GhKh} can instead be written
\begin{align}
\Qh{\w} &= \Bh{} \triangleq \Sc \inv{\timestack{{\bar Q_{\wb{}}}}}\Sc^T\label{eq:parric:Qh_Bh_alternative_def} \textrm{  and   } \; 
\Gh{}\hat L = - \Hh{}^T, 
\end{align}
where $\Gh{}$ and $\Hh{}$ are defined as in~\eqref{eq:proof:thm:Gh_Hh_def} but with the new $\Qh{\w}$ and $\Bh{}$.
The \uftoc problem corresponding to~\eqref{eq:parric:reduce} then obtains an (possibly) increased control signal dimension $n_{\hat w} = n_x \geq \nuo$ compared to when $\Qh{\w}$ and $\Bh{}$ are defined as in~\eqref{eq:Bh_def}, but with the advantage that $\Qh{\w}$ and $\Bh{}$ can be easily computed. Analogous calculations can be made for $\kh{}$.

\begin{remark}
If $\Sc^T$ is rank deficient then $U_1\in \mathbb{R}^{N\nw \times \nuo}$ has $\nuo < \nx$ columns. Hence $\Gh{}$ is singular and $\hat L$ non-unique in~\eqref{eq:parric:Qh_Bh_alternative_def}. How to cope with this is described in, \eg,~\cite{axehill08:thesis,nielsen15:licthesis}.
\end{remark}

For the last subproblem with $i=p$, the variables $\Prh{p+1} = \Qx{N_p,p}$, $\Psirh{p+1} = - \linx{N_p,p}$ and $\crh{p+1} = c_{N_p,p}$ in~\eqref{eq:sub_problem_dscr} are in fact known. Hence, the last subproblem can be factored directly and all variables but the initial state can be eliminated.

The formal validity of the reduction of each subproblem $i \in \intset{0}{p-1}$ is given by Theorem~\ref{thm:reduce_subprob}, while the computational procedure is summarized in Algorithm~\ref{alg:prel_fact}, which is basically a Riccati factorization and backward recursion as in algorithms~\ref{alg:factorization} and~\ref{alg:bwd_rec}. Here $\Qwh{i}$ and $\Bh{i}$ are computed as in~\eqref{eq:parric:Qh_Bh_alternative_def}. 
\vspace{-15pt}
\begin{algorithm}[h!]
  \caption{Reduction using Riccati factorization} \label{alg:prel_fact}
  \begin{algorithmic}[1]
    \STATE $P_N :=   0, \; \Psi_N :=  0, \; \cb{N} :=  0$ \\
    $\hat Q_w := 0, \; \Vk{N} := I, \; \vk{N} := 0 $  \label{alg:prel_fac:line:init}
    \FOR{$t=N-1,\ldots,0$}
    \STATE $F_{t+1} := Q_{x,t} + A^T_tP_{t+1}A_t$\label{alg:prel_fact:line:F_comp} \\
    \STATE $G_{t+1} := Q_{\w,t} + B^T_tP_{t+1}B_t$ \\
    \STATE $H_{t+1} := Q_{x\w,t} + A^T_tP_{t+1}B_t$\label{alg:prel_fact:line:H_comp} \\
    \STATE Compute and store a factorization of $G_{t+1}$.\label{alg:prel_fact:line:factorize_G}
    \STATE Compute a solution $K_{t+1}$ to: 
    $G_{t+1}K_{t+1} = -H^T_{t+1}$ \
    \STATE Compute a solution $k_{t+1}$ to \\
     $G_{t+1}k_{t+1} = B^T_t\Psi_{t+1} -
        \lin{\w,t} - B_t^T P_{t+1}a_t$ \\
    \STATE $\Psi_t := A^T_t\Psi_{t+1}
    - H_{t+1}k_{t+1} - \lin{x,t} - A_t^T P_{t+1}a_t$
    \STATE $P_{t} := F_{t+1} - K^T_{t+1}G_{t+1}K_{t+1}$\label{alg:prel_fact:line:P_comp}
    \STATE Compute a solution $L_{t+1}$ to:  
    $G_{t+1}L_{t+1} = -B_t^T \Vk{t+1}$ \\
    \STATE $\Vk{t} := \left( A_t^T + K_{t+1}^T B_t^T \right)\Vk{t+1}$
    \STATE $\vk{t} := \vk{t+1} + \Vk{t+1}^T(a_t+B_tk_{t+1})$
    \STATE $\hat Q_\w := \hat Q_\w + L_{t+1}^T G_{t+1} L_{t+1}$
    \ENDFOR
    \STATE $\Ah{} := \Vk{0}^T, \; \Bh{} := \Qh{\w}, \; \ah{} := \vk{0}$ \\
  	$ \Qh{x} := P_0, \; \lh{x} := -\Psi_0, \; \ch{} := \cb{0} $
  \end{algorithmic}
\end{algorithm}



\subsection{Constructing the master problem}
\label{subsec:red_mpc_prob}
According to Theorem~\ref{thm:reduce_subprob} and the theory presented in Section~\ref{subsec:red_sub_prob}, all subproblems $i \in \intset{0}{p-1}$ can be reduced to depend only on the variables $\xh{i}$, $x_{N_i,i}$ and $\wh{i}$, and subproblem $i=p$ depends only on $\xh{p}$. The variable $\wh{i}$ represents the unknown part of $\w_{t,i}$ that are due to the initially unknown $\Ph{i+1}$ and $\Psih{i+1}$. Using the definition of the subproblems and the property ${x_{N_i,i} = x_{0,i+1} = \xh{i+1}}$ that were introduced in the decomposition in Section~\ref{subsec:sub_prob}, the reduced subproblems $i \in \intset{0}{p}$ can be combined into a master problem which is equivalent to the problem in~\eqref{eq:org_eqc_problem}. By using the notation from Section~\ref{subsec:red_sub_prob}, the master problem can be written
\begin{equation}
\minimizes{&\sum_{i=0}^{\p-1}\left(\frac{1}{2}\begin{bmatrix}
\xh{i} \\ \wh{i}
\end{bmatrix}^T    
\Qh{i} \begin{bmatrix}
    \xh{i} \\  \wh{i}
\end{bmatrix} + \lh{x,i}^T \xh{i} + \ch{i} \right) \\&+ \frac{1}{2} \xh{\p}^T \Qxh{\p} \xh{\p} + \lh{x,p}^T\xh{p}+\ch{p}}{\timestack{\xh},\timestack{\wh}}{\xh{0}&=\bar x_0 \\ \xh{i+1} &= \Ah{i} \xh{i} + \Bh{i} \wh{i} + \ah{i}, \; i \in \intset{0}{\p-1}.}
\label{eq:red_mpc_problem}
\end{equation}
This is a \uftoc problem in the same form as~\eqref{eq:org_eqc_problem} but with shorter prediction horizon $p <N$ and block-diagonal $\Qh{i}$. The dynamics equations $\xh{i+1}=\Ah{i}\xh{i} + \Bh{i} \wh{i} + \ah{i}$ are due to the relation
\begin{equation}
\xh{i+1} = x_{0,i+1} = x_{N_i,i} = \Ah{i}\xh{i} + \Bh{i} \wh{i} + \ah{i}.
\end{equation}
 Hence, a \uftoc problem $\MPCprob{N}$ can be reduced to a \uftoc problem $\MPCprob{\p}$ in the same form but with shorter prediction horizon and possibly fewer variables dimension in each time step. Each subproblem is reduced individually using an algorithm based on the Riccati recursion.
%


\section{Computing the  Riccati Recursion in Parallel}
\label{sec:par_ric}
The reduction of the individual subproblems according to Section~\ref{subsec:red_sub_prob} can be performed in parallel. To reach consensus between all subproblems in order to solve the original problem~\eqref{eq:org_eqc_problem}, the master problem~\eqref{eq:red_mpc_problem} can be solved to obtain $\Ph{i+1}$, $\Psih{i+1}$, $\ch{i+1}$ and the optimal $\xh{i}$ for $i \in \intset{0}{\p}$. When these variables are computed, the independent subproblems can be solved in parallel using algorithms~\ref{alg:factorization}-\ref{alg:fwd_rec} with the initial $x_{0,i} = \xh{i}$, $\Ph{i+1}$, $ \Psih{i+1}$ and $ \ch{i+1}$ for $i\in \intset{0}{\p}$. 


To compute $\Ph{i+1}$, $\Psih{i+1}$, $\ch{i+1}$ and $\xh{i}$ the master problem~\eqref{eq:red_mpc_problem} can be solved serially using the Riccati recursion. However,~\eqref{eq:red_mpc_problem} can instead itself be reduced in parallel in an upward pass until a \uftoc problem with a prediction horizon of pre-determined length is obtained. This top problem is then solved, and the solution is propagated down in the tree in a downward pass until the subproblems of the original problem~\eqref{eq:org_eqc_problem} are solved. This procedure is shown in Fig.~\ref{fig:arb_tree_struct}, where $\MPCsub{i}{k}$ denotes subproblem $i$ in the form~\eqref{eq:sub_problem_dscr} at level $k$ in the tree. The number of steps in the upward and downward pass are known a-priori and can be determined by the user.

Since the subproblems at each level can be reduced and solved in parallel and the information flow is between parent and children in the tree, the Riccati recursion can be computed in parallel using the theory proposed in this paper.

\begin{figure}[htb]
\centering
\def\svgwidth{\columnwidth}
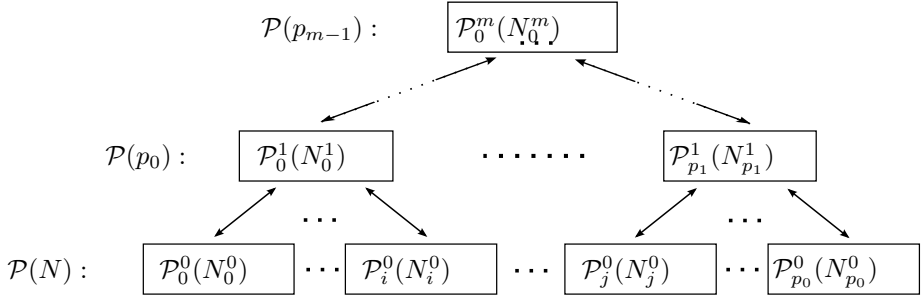
\caption{The original \uftoc problem $\MPCprob{N}$ can be reduced repeatedly using Riccati recursions. When the solution to the top problem is computed, it can be propagated back in the tree until the bottom level is solved. }
\label{fig:arb_tree_struct}
\end{figure}

%


The \uftoc problem~\eqref{eq:org_eqc_problem} is reduced in parallel in several steps in Algorithm~\ref{alg:build_tree} to a \uftoc problem with prediction horizon~$\p_{m-1}$. Assume, for simplicity, that all subproblems are of equal batch length $N_s$ and that $N=N_s^{m+1}$ for some integer $m \geq 1$. Then, provided that $N_s^m$ computational units are available, the reduction can be made in $m$ steps, \ie, the reduction algorithm has $\Ordo{\log N}$ complexity growth.
\begin{algorithm}[h!]
  \caption{Parallel reduction of \uftoc problem} \label{alg:build_tree}
  \begin{algorithmic}[1]
  	\STATE Set the maximum level number $m$
  	\STATE Set the number of subproblems $p_k+1$ for each level $k \in \intset{0}{m}$ with $p_m=0$
 	\FOR{$k:=0,\ldots,m-1$} 
   	    \STATE \textbf{parfor} $i = 0,\ldots,p_k$ \textbf{do} \label{alg:line:build_tree_loop}
   	    \STATE \hspace{1.5ex} Create subproblem $\MPCsub{i}{k}$
        	\STATE \hspace{1.5ex} Reduce subproblem $\MPCsub{i}{k}$ using Algorithm~\ref{alg:prel_fact}
      		\STATE \hspace{1.5ex} Send $\Ah{i}^k$, $\Bh{i}^k$, $\ah{i}^k$, $\Qxh{i}^k$, $\Qwh{i}^k$, $\lh{x,i}^k$ and $\ch{i}^k$ to parent 
     	\STATE \textbf{end parfor}
    \ENDFOR
  \end{algorithmic}
\end{algorithm}
%

When the reductions of the subproblems are completed, Algorithm~\ref{alg:propagate_solution} is applied to solve each subproblem $i \in \intset{0}{p_{k-1}}$ at level $k$ using algorithms~\ref{alg:factorization}-\ref{alg:fwd_rec} with the optimal $\xh{i}^{k+1}$, $\hat P_{i+1}^{k+1}$, $\Psih{i+1}^{k+1}$ and $\ch{i+1}^{k+1}$ from the respective parent. The algorithm starts by solving the top problem $\MPCprob{\p_{m-1}}$ in Fig.~\ref{fig:arb_tree_struct}, and the solution is passed to its children. By solving the subproblems at each level and passing the solution to the children at the level below in the tree, the subproblems $\MPCsub{i}{0}$, $i \in \intset{0}{\p_0}$ at the bottom level can finally be solved individually. All subproblems can be solved using only information from their parents, and hence each level in the tree can be solved in parallel. 

By using the definition of the local variables, the optimal primal solution to the original \uftoc problem~\eqref{eq:org_eqc_problem} can be constructed from the solutions to the subproblems at the bottom level. The dual variables can be computed from all $P_{t,i}^0$ and $\Psi_{t,i}^0$ from the subproblems at the bottom level. Hence there are no complications with non-unique dual variables as in~\cite{nielsen14:parallel_mpc_ifac} when using the algorithm presented in this paper. The propagation of the solution from the top level to the bottom level can be made in $m+1$ steps provided that $N_s^m$ processing units are available. Since both algorithms~\ref{alg:build_tree} and~\ref{alg:propagate_solution} are solved in $\Ordo{\log N}$ complexity, the Riccati recursion and the solution to~\eqref{eq:org_eqc_problem} can be computed with $\Ordo{\log N}$ complexity growth. 
\begin{algorithm}
  \caption{Parallel solution of \uftoc problem} \label{alg:propagate_solution}
  \begin{algorithmic}[1]
  	\STATE Get the top level number $m$ and $p_k$:s from Algorithm~\ref{alg:build_tree}
  	\STATE Initialize $\xh{0}^{m} := \xinit$
    \FOR{$k := m, m-1,\ldots, 0$}
	    \STATE \textbf{parfor} {$i:=0,\ldots,p_k$} \textbf{do}\label{alg:line:propagate_solution_loop}
			\STATE \hspace{1.5ex} Solve subproblem $\MPCsub{i}{k}$ using algorithms~\ref{alg:factorization}-\ref{alg:fwd_rec}\label{alg:propagate_solution:line:solve}
		\STATE \hspace{1.5ex} \textbf{if }{$k > 0$} \textbf{then}
    		\STATE \hspace{3.5ex} Send $\Ph{t,i}^k, \Psih{t,i}^k, \ch{t,i}^k$ and $\xh{t,i}^k$ to each children\label{alg:line:prop_sol_send}
    		\STATE \hspace{1.5ex} \textbf{end if}
    	\STATE \textbf{end parfor} \label{alg:line:propagate_solution_loop_end}
    \ENDFOR
    \STATE Get the solution of~\eqref{eq:org_eqc_problem} from the solutions of all $\MPCsub{i}{0}$ 
  \end{algorithmic}
\end{algorithm}

Beyond what is presented here, as was observed already in~\cite{axehill:_towar_mpc}, standard parallel linear algebra can be used for many computations in the serial Riccati recursion in each subproblem to boost performance even further. This has however not been utilized in this work.




\section{Numerical results}
\label{subsec:num_res}

In the presented experiments from \textsc{Matlab}, parallel executions of the algorithms are simulated by executing them serially using one computational thread but still using the same information flow as for an actual parallel execution. The total computation time has been estimated by summing over the maximum computation time for each level in the tree, and hence the communication overhead is neglected. The influence of the communication overhead is discussed in the end of this section. The performance when computing the solution to~\eqref{eq:org_eqc_problem} of the parallel Riccati algorithm proposed in this work is compared with both the serial Riccati recursion and, as a reference for linear problems, the well-known \RTS smoother (see, \eg{},~\cite{kailath2000linear}). The \RTS smoother is only implemented in \matlab.
In all results presented in this section $N_s = 2$ has been used in the parallel algorithm.
\begin{remark}
Different batch lengths can be used for each subproblem in the tree. How to choose these to minimize computation time is not investigated here. However, similarly as in~\cite{axehill15:_contr_mpc}, the optimal choice depends on, \eg, the problem and the hardware on which the algorithm is implemented. 
\end{remark}

In \matlab the algorithms have been compared when solving \MHE problems (or computing Newton steps for inequality constrained \MHE problems) in the form~\eqref{eq:org_mhe} for systems of dimension $n_x = 20$, $\nw=20$ and $\ny=20$, see Fig.~\ref{fig:comp_log_nx20_nw20_ny20}. It can be seen that the parallel Riccati algorithm outperforms the serial Riccati for $N \gtrsim 20$ and the \RTS smoother for $N \gtrsim 30$. 

\begin{figure}[htb!]
\centering
\includegraphics[width=\columnwidth]{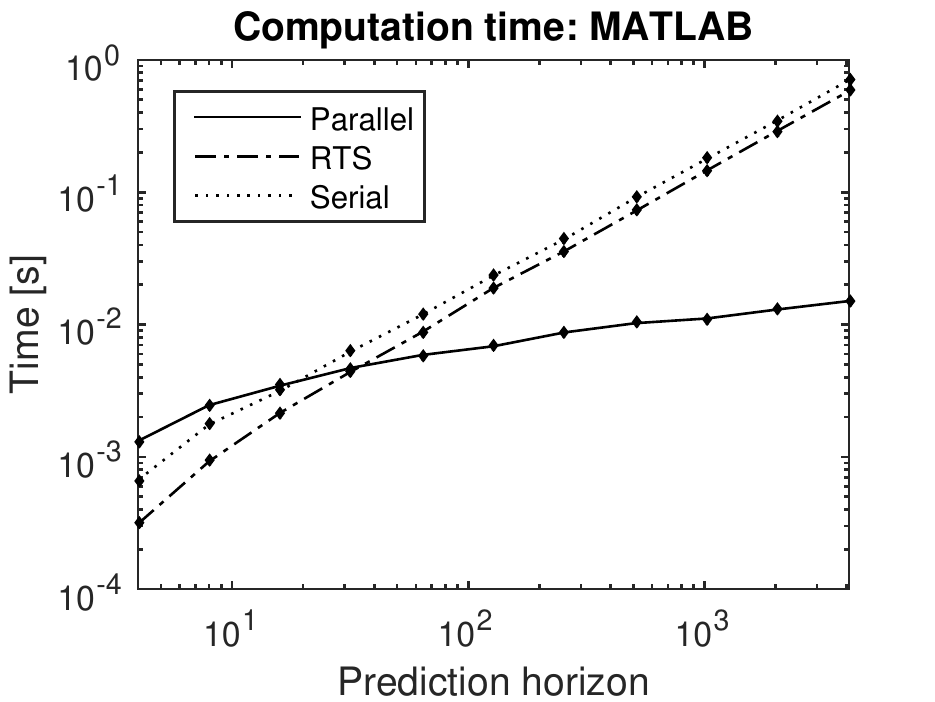}
\caption{Computation times when solving \MHE problems of order $n_x=20$, $\nw=20$ and $\ny=20$. The parallel Riccati algorithm outperforms the serial Riccati algorithm for $N\gtrsim 20$ and the \RTS smoother for $N \gtrsim 30$. }
\label{fig:comp_log_nx20_nw20_ny20}
\end{figure}
An \ansic implementation has been run on a computer cluster consisting of nodes with 8-core Intel Xeon E5-2660 @ 2.2 GHz \cpus with communication over \tcpip on Gigabit Ethernet. The computations were performed on resources provided by the Swedish National Infrastructure for Computing (\snic) at \nsc. The implementation is rudimentary and especially the communication setup can be improved, but the implemented algorithm serves as a proof-of-concept that the algorithm improves performance in terms of computation times for computations on real parallel hardware, taking communication delays into account. The computation times when solving \MHE problems in the form~\eqref{eq:org_eqc_problem} for systems of order $n_x = 20$, $\nw=20$ and $\ny=20$ are seen in Fig.~\ref{fig:comp_log_nx20_nu20_ansic}, where it is clear that the parallel algorithm solves a problem with $N=512$ approximately as fast as the serial algorithm solves the same one for $N=64$, and the break even is at $N \approx 24$. This computational speed-up can be important in smoothing problems and in \MHE problems where long horizons are often used,~\cite{rao98:_applic_inter_point_method_model_predic_contr}.

The communication overhead is approximately $20\%$ for this problem size, and it has been observed that communication times are roughly the same regardless of problem size. This indicates that there is a significant communication latency, and reducing these can significantly improve performance of the \ansic implemented algorithm.

\begin{figure}[htb!]
\centering
\includegraphics[width=\columnwidth]{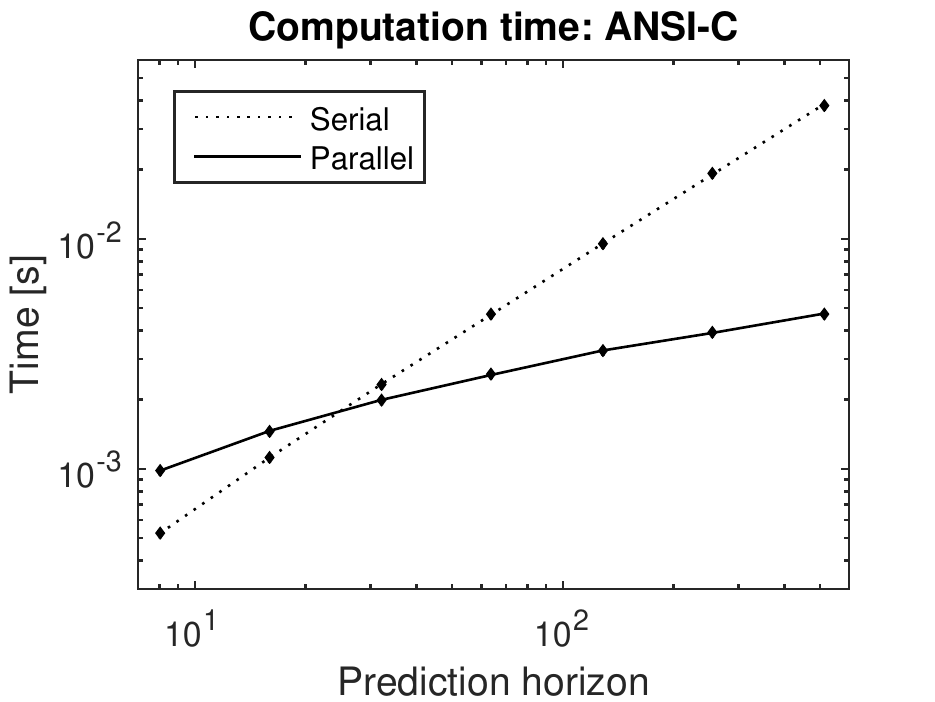}
\caption{Computation times for \ansic implementation using problems of order $n_x=20$, $\nw=20$ and $\ny = 20$. The parallel Riccati algorithm outperforms the serial for $N \gtrsim 24$ and it computes the solution to an \MHE problem with $N=512$ approximately as fast as for $N=64$ using the serial algorithm.}
\label{fig:comp_log_nx20_nu20_ansic}
\end{figure}

\section{Conclusions}
\label{sec:conclusion}
In this paper it is shown that the Newton step that is required in many methods for solving \MHE problems can be computed directly (non-iteratively) in parallel using Riccati recursions that exploit the structure from the \MHE problem. The proposed algorithm obtains logarithmic complexity growth in the estimation horizon length. Results from numerical experiments both in \matlab as well as in \ansic on real parallel hardware show that the proposed algorithm outperforms existing serial algorithms already for relatively small values of $N$. Future work includes the possibility to improve performance and reduce communication latencies by using more suitable hardware such as, \eg, Graphics Processing Units (\gpus) or Field-Programmable Gate Arrays (\fpgas).


\bibliography{axe_full,IEEEfull,ianFull}

\end{document}